\documentclass[10pt, a4paper]{amsart}
\usepackage{amsmath,amssymb,amsthm,enumerate,indentfirst}

\usepackage[top=30truemm,bottom=30truemm,left=30truemm,right=30truemm]{geometry}

\newtheorem{theorem}{Theorem}[section]
\newtheorem{definition}[theorem]{Definition}
\newtheorem{proposition}[theorem]{Proposition}
\newtheorem{lemma}[theorem]{Lemma}
\newtheorem{conjecture}[theorem]{Conjecture}

\theoremstyle{remark}
\newtheorem{remarks}[theorem]{Remarks}

\newcommand{\card}{\,^\#\!}
\newcommand{\Drin}{\mathrm{Drin}}
\newcommand{\End}{\mathrm{End}}
\newcommand{\Gal}{\mathrm{Gal}}
\newcommand{\GL}{\mathrm{GL}}
\newcommand{\sep}{\mathrm{sep}}
\newcommand{\tor}{\mathrm{tor}}

\begin{document}

\title{Torsion points of Drinfeld modules over large algebraic extensions of finitely generated function fields}
\author{Takuya Asayama}
\date{}
\maketitle

\begin{abstract}
Geyer and Jarden proved several results for torsion points of elliptic curves defined over the fixed field by finitely many elements in the absolute Galois group of a finitely generated field over the prime field in its algebraic closure.
As an analogue of these results, this paper studies torsion points of Drinfeld modules defined over the fixed field by finitely many elements in the absolute Galois group of a finitely generated function field in its algebraic closure.
We prove some results which are similar to those of Geyer and Jarden.
\end{abstract}

\pagestyle{myheadings}
\markboth{TAKUYA ASAYAMA}{TORSION POINTS OF DRINFELD MODULES}

\renewcommand{\thefootnote}{\fnsymbol{footnote}}
\renewcommand{\thefootnote}{\arabic{footnote}}
\renewcommand{\thefootnote}{\fnsymbol{footnote}}
\renewcommand{\thefootnote}{\arabic{footnote}}

\section{Introduction}

Geyer and Jarden \cite{GJ1} considered torsion points of elliptic curves defined over some class of infinitely generated fields over their prime fields and proved several results.
They also conjectured that the same results hold even when considering arbitrary abelian varieties rather than elliptic curves.  
Our goal of this paper is to prove some results for Drinfeld modules analogous to the conjecture of Geyer-Jarden.

Let $K$ be a field.
We fix an algebraic closure $\tilde{K}$ of $K$, and denote by $K^\sep$ the separable closure of $K$ in $\tilde{K}$.
Let $e$ be a positive integer.
For every $e$-tuple $\sigma = (\sigma_1 , \sigma_2 , \dots , \sigma_e)\in \Gal(K^\sep/K)^e$, write $\tilde{K}(\sigma)$ (resp. $K^\sep(\sigma)$) for the fixed field in $\tilde{K}$ (resp. $K^\sep$) by $\sigma$. (We extend each $\sigma_i \in \Gal(K^\sep/K)$ to $\tilde{K}$ uniquely.)
We use the term \textit{almost all} in the sense of the Haar measure defined on the profinite group $\Gal(K^\sep/K)^e$.

Frey and Jarden proved the following theorem for the rank of abelian varieties defined over $K^\sep(\sigma)$.

\begin{theorem}[Frey-Jarden{~\cite[Theorem 9.1]{FyJ}}]
Let $K$ be a finitely generated infinite field over its prime field.
Then almost all $\sigma \in \Gal(K^\sep/K)^e$ have the following property\textup{:}
for any abelian variety $A$ of positive dimension defined over $K^\sep(\sigma)$ the group $A(K^\sep(\sigma))$ of $K^\sep(\sigma)$-rational points of $A$ has infinite rank.
\end{theorem}

The work of Geyer and Jarden \cite{GJ1} is on the torsion subgroup of $E(\tilde{K}(\sigma))$, the group of $\tilde{K}(\sigma)$-rational points of an elliptic curve $E$ over $\tilde{K}(\sigma)$. In contrast to the above theorem, some of their results depend on whether $e$ equals to one or not.

\begin{theorem}[Geyer-Jarden{~\cite[Theorem 1.1]{GJ1}}]\label{GJ}
Let $K$ be a finitely generated field over its prime field. Then for almost all $\sigma \in \Gal(K^\sep/K)^e$ and for every elliptic curve $E$ defined over $\tilde{K}(\sigma)$ the following statements hold\textup{:}
\textup{
    \begin{enumerate}[(a)]
    \item \textit{If $e = 1$, then the group $E_\tor(\tilde{K}(\sigma))$ is infinite. Moreover, there exist infinitely many primes $l$ such that $E(\tilde{K}(\sigma))[l] \neq 0$.}
    \item \textit{If $e \ge 2$, then the group $E_\tor(\tilde{K}(\sigma))$ is finite.}
    \item \textit{If $e \ge 1$, then for every prime $l$ the group $E(\tilde{K}(\sigma))[l^\infty] = \bigcup_{i = 1}^\infty E(\tilde{K}(\sigma))[l^i]$ is finite.}
    \end{enumerate}
}
\noindent Here $E(\tilde{K}(\sigma))[n]$ is the group of $n$-torsion points in $E(\tilde{K}(\sigma))$.
\end{theorem}

The difference between the two cases $e  = 1$ and $e \ge 2$ can be explained to some extent by the fundamental fact that the sum $\sum_l l^{-e}$, where $l$ ranges over all primes, diverges for $e = 1$ and converges for $e \ge 2$.
Note that the measure of the set of $\sigma \in \Gal(K^\sep/K)^e$ such that $E(\tilde{K}(\sigma))[l] \neq 0$ is approximated by $l^{-e}$ up to multiplication by some positive constant unless the $j$-invariant of $E$ is contained in a finite field.

Theorem 1.2 is naturally generalized to the conjecture below as indicated in the same paper.

\begin{conjecture}[Geyer-Jarden{~\cite[Conjecture]{GJ1}}]
Theorem \ref{GJ} is still true even if one replaces $E$ by an arbitrary abelian variety $A$ of positive dimension.
\end{conjecture}

Jacobson and Jarden proved in \cite{JJ1} that the conjecture is true if $K$ is a finite field.
In the same paper, there is a proof of Part (a) of the conjecture for $K$ with positive characteristic, but the proof contains an error as indicated in \cite{JJ2}.
They showed in \cite{JJ3} that Part (c) always holds and that Part (b) holds if $K$ has characteristic zero.
The case of Part (a) for characteristic zero appears in a paper of Geyer and Jarden \cite{GJ2} and another one of Zywina \cite{Zy}.
Recently, Jarden and Petersen \cite{JP} completed the proof of the conjecture for this case.
Parts (a) and (b) for $K$ which is infinite and has positive characteristic are still open.

We consider an analogue of Geyer-Jarden's conjecture for Drinfeld modules.
Denote by $\mathbb{F}_q$ the finite field with order $q$. Let $K$ be an algebraic function field in one variable over $\mathbb{F}_q$.
Fix a place $\infty$ of $K$, and let $A$ be the ring consisting of the elements of $K$ which are regular outside $\infty$.
Let $L$ be a field containing $\mathbb{F}_q$ and fix an injective ring homomorphism $\iota: A \to L$ (i.e., we consider an $A$-field with generic characteristic).
Our main result is the following:

\begin{theorem}\label{mainthm}
Suppose that $L$ is finitely generated over $K$.
Then for almost all $\sigma \in \Gal(L^\sep/L)^e$ and for every Drinfeld $A$-module $\varphi$ defined over $\tilde{L}(\sigma)$ with $\End_{\tilde{L}} \varphi = A$ the following statements hold\textup{:}
\textup{
    \begin{enumerate}[(a)]
    \item \textit{Assume $e = 1$. The group $_\varphi\tilde{L}(\sigma)_\tor$ is infinite. Moreover, there exist infinitely many non-zero prime ideals $P$ of $A$ such that $_\varphi\tilde{L}(\sigma)[P] \neq 0$.}
    \item \textit{Assume $e \ge 2$. The group $_\varphi\tilde{L}(\sigma)_\tor$ is finite.}
    \item \textit{The group $_\varphi\tilde{L}(\sigma)[P^\infty] = \bigcup_{n = 1}^\infty {}_\varphi\tilde{L}(\sigma)[P^n]$ is finite for every non-zero prime ideal $P$ of $A$.}
    \end{enumerate}
}
\end{theorem}

\begin{remarks}
1. It turns out that the difference between the two cases $e  = 1$ and $e \ge 2$ arises from the fact that the sum $\sum_P N(P)^{-e}$, where $P$ runs over all prime ideals of $A$ and $N(P) = \card(A/P)$, diverges for $e = 1$ and converges for $e \ge 2$.

\noindent 2. This analogue of Geyer-Jarden's conjecture for Drinfeld modules can also be considered in the case where $\iota$ is not injective (such a field is said to be an $A$-field having \textit{finite characteristic} or \textit{special characteristic}).
The problem in this case remains open.
\end{remarks}

The rest of this paper is organized as follows.
In Section 2, we recall the Haar measure of a profinite group.
Section 3 is devoted to review of the definition of Drinfeld modules and their known properties.
In Section 4, we prove that in order to prove Theorem \ref{mainthm} it suffices to show some statements on \textit{each} Drinfeld module over $L$.
Finally, in Section 5, we derive these statements and complete the proof of the main theorem.

\section{The Haar measure of a profinite group}

In this section we recall the definition of the Haar measure of a profinite group and its fundamental properties used in the later sections, especially in the case where the profinite group arises from a Galois extension.

Let $G$ be a profinite group and $\mathcal{B}$ be its Borel algebra.
A function $\mu: \mathcal{B} \to \mathbb{R}$ is a (\textit{normalized}) \textit{Haar measure} if $\mu$ is a probability measure and satisfies the following extra conditions:

\begin{enumerate}[(1)]
\item (\textit{translation invariance}) If $B \in \mathcal{B}$ and $g \in G$, then $\mu(g B) = \mu(B g) = \mu(B)$.
\item (\textit{regularity}) For $B \in \mathcal{B}$ and $\varepsilon > 0$ there exist an open set $U$ and a closed set $C$ in $G$ such that $C \subseteq B \subseteq U$ and $\mu(U \setminus C) < \varepsilon$.
\end{enumerate}

It is known that a Haar measure is uniquely defined on every profinite group $G$ (see \cite[Proposition 18.2.1]{FrJ}).
We write this measure by  $\mu$, or by $\mu_G$ if a reference to $G$ is needed.
We use the same notation for its completion.

A family $\{ A_i \}_{i \in I}$ of measurable sets of $G$ is said to be \textit{$\mu$-independent} if $\mu(\bigcap_{i \in J} A_i) = \prod_{i \in J} \mu(A_i)$ for every finite subset $J \subseteq I$.

If $G$ and $H$ are profinite groups, then the direct product $G \times H$ is also a profinite group.
Thus we can equip $G \times H$ with the Haar measure $\mu_{G \times H}$.
On the other hand, we can also consider the \textit{product measure} $\mu_G \times \mu_H$ of $G \times H$, which satisfies $(\mu_G \times \mu_H)(A \times B) = \mu_G(A) \mu_H(B)$ for all measurable subsets $A \subseteq G$ and $B \subseteq H$.
It can be proved that $\mu_{G \times H}$ and $\mu_G \times \mu_H$ coincide (after completion) \cite[Proposition 18.4.2]{FrJ}.
Apply this result to $G^e$, the product of $e$ copies of $G$.
We simply denote by $\mu$ or $\mu_G$ again the Haar measure $\mu_{G^e} = \mu_G \times \mu_G \times \dotsm \times \mu_G$ ($e$ times) of $G^e$.

\begin{lemma}[Borel-Cantelli's lemma, see {\cite[Lemma 18.3.5]{FrJ}}]
Let $\{ A_i \}_{i=1}^\infty$ be a countable collection of measurable subsets of a profinite group $G$. Define
\[ A = \bigcap_{n=1}^\infty \bigcup_{i=n}^\infty A_i = \{ g \in G \mid g \in A_i \; \text{for infinitely many} \; i \! \text{'s} \}. \]
Then we have the following\textup{:}
\textup{
\begin{enumerate}[(1)]
    \item \textit{If $\sum_{i=1}^\infty \mu(A_i) < \infty$, then $\mu(A) = 0$.}
    \item \textit{If $\{ A_i \}_{i=1}^\infty$ is $\mu$-independent and $\sum_{i=1}^\infty \mu(A_i) = \infty$, then $\mu(A) = 1$.}
\end{enumerate}
}
\end{lemma}

Let $K$ be a field and $M$ be a Galois extension of $K$.
We equip the Galois group $\Gal(M/K)$ with a structure of profinite group arising from the \textit{Krull topology}.
The Krull topology is induced by the isomorphism $\Gal(M/K) \to \varprojlim_{N \in \mathcal{N}} \Gal(N/K)$ given by $\sigma \mapsto (\sigma|_N)_{N \in \mathcal{N}}$, where $\mathcal{N}$ is the set of all finite Galois extensions of $K$ in $M$.
We give each $\Gal(N/K)$, $N \in \mathcal{N}$ the discrete topology, and each map $\Gal(N'/K) \to \Gal(N/K)$, $N, N' \in \mathcal{N}$, $N \subseteq N'$ is just a restriction.
It is compact, Hausdorff, and the family $\{ \Gal(M/N) \mid N \in \mathcal{N} \}$ is a basis for the open neighborhoods of the identity in $\Gal(M/K)$.

A family $\{ M_i \}_{i \in I}$ of finite Galois extensions of $K$ is said to be \textit{linearly disjoint} over $K$ if $[M_J : K] = \prod_{i \in J} [M_i : K]$ for every finite subset $J \subseteq I$, where $M_J$ is the composition of $M_i$, $i \in J$.
Note that a family $\{ M_i \}_{i \in I}$ of finite Galois extensions of $K$ is linearly disjoint over $K$ if and only if $\{ \Gal(K^\sep/M_i) \}_{i \in I}$ is $\mu$-independent in $\Gal(K^\sep/K)$ (see \cite[Lemma 18.5.1]{FrJ}).


\section{Drinfeld modules}

The aim of this section is to recall the definition of Drinfeld modules.

Let $K$ be an algebraic function field in one variable over $\mathbb{F}_q$. We fix a place $\infty$ of $K$.
Denote by $A$ the ring of functions in $K$ which are regular outside $\infty$.
An \textit{$A$-field} is a field $L$ which is equipped with a homomorphism of $\mathbb{F}_q$-algebras $\iota: A \to L$.
Throughout this paper, we consider only $A$-fields with $\iota$ injective.
Such an $A$-field is said to have \textit{generic characteristic}.
For an $A$-field $L$ with generic characteristic we can identify $A$ with the image of $\iota$.
Thus it allows us to consider $L$ as an extension of $K$.

Let $L\{\tau\}$ be a twisted polynomial ring over $L$ generated by the $q$-th power Frobenius morphism $\tau$, with the relation $\tau a = a^q \tau$ for all $a \in L$.
Let $D: L\{\tau\} \to L$ be the morphism which takes a twisted polynomial to its constant term.

\begin{definition}
A Drinfeld $A$-module over $L$ is a ring homomorphism of $\mathbb{F}_q$-algebras $\varphi: A \to L\{\tau\}$ such that $D \circ \varphi = \iota$ and there is $a \in A$ with $\varphi_a \neq \iota(a)$.
If $A$ and $\iota$ are fixed, then we denote by $\Drin_A L$ the set of all Drinfeld $A$-modules over $L$.
\end{definition}

Let $L'$ be a field extension of $L$.
Each Drinfeld $A$-module $\varphi$ over $L$ defines a structure of an $A$-module in $L'$ by
\[ a u = \varphi_a (u), \quad a \in A, \, u \in L'. \]
Here $f(x) = \sum_i a_i x^{q^i}$ for $f = \sum_i a_i \tau^i \in L\{\tau\}$ and $x \in L'$. 
We use the notation $_\varphi L'$ when we regard $L'$ as an $A$-module induced by $\varphi$ in the above way.

For any $a \in A$, $a \neq 0$, define the \textit{$a$-torsion submodule} $_\varphi L'[a]$ of $_\varphi L'$ by
\[ _\varphi L'[a] = \{ u \in L' \mid \varphi_a(u) = 0 \}. \]
Let $I$ be a non-zero ideal of $A$. Then the \textit{$I$-torsion submodule} $_\varphi L'[I]$ of $_\varphi L'$ is defined by $_\varphi L'[I] = \bigcap_{0 \neq a \in I} {}_\varphi L'[a]$.
We immediately find that if $I = (a)$ is a principal ideal, then $_\varphi L'[I] = {}_\varphi L'[a]$.
The \textit{$I$-power torsion submodule} $_\varphi L'[I^\infty]$ is said to be $_\varphi L'[I^\infty] = \bigcup_{n=1}^\infty {}_\varphi L'[I^n]$.
Finally, the \textit{torsion submodule} $_\varphi L'_\tor$ is defined by $_\varphi L'_\tor = \bigcup_{0 \neq a \in A} {}_\varphi L'[a] = \bigcup_{0 \neq I \subseteq A} {}_\varphi L'[I]$.
That is, $_\varphi L'_\tor$ consists of all $u \in {}_\varphi L'$ such that there is $a \in A$, $a \neq 0$ with $\varphi_a(u) =0$.

We can show that for every Drinfeld $A$-module $\varphi$ over $L$ there exists a positive integer $r$ satisfying $\deg_\tau \varphi_a = r \deg a$ for each $a \in A$.
This integer $r$ is said to be the \textit{rank} of $\varphi$.

\begin{theorem}[see {\cite[Theorem 13.1 and its Corollary]{Ro}}]\label{I-tor}
Let $\varphi$ be a Drinfeld $A$-module over $L$ with rank $r$ and $n$ be a positive integer.
Then for any non-zero prime ideal $P$ of $A$, we have
\[ _\varphi\tilde{L}[P^n] \cong (A/P^n)^r. \]
More generally, for any non-zero ideal $I$ of $A$, we have
\[ _\varphi\tilde{L}[I] \cong (A/I)^r. \]
\end{theorem}

The next result is an analogue of the Mordell-Weil theorem. Poonen proved for the case where $L$ is a finite extension over $K$, and Wang extended this theorem to an arbitrary finitely generated extension $L$ over $K$.

\begin{theorem}[Poonen{~\cite[Theorem 1]{Po}}, Wang{~\cite[Theorem 1]{Wa}}]
Let $L$ be a finitely generated extension over $K$ and $\varphi$ be a Drinfeld $A$-module over $L$. Then the group $_\varphi L$ is the direct sum of a finite torsion submodule $_\varphi L_\tor$ and a free $A$-module of rank $\aleph_0$. 
\end{theorem}

We recall the notions of a morphism of Drinfeld modules and the ring of endomorphisms of a Drinfeld module.

\begin{definition}
Let $\varphi$, $\psi$ be Drinfeld $A$-modules over $L$, and $L'$ be a field extension of $L$. A morphism from $\varphi$ to $\psi$ over $L'$ is a twisted polynomial $f \in L'\{ \tau \}$ satisfying $f \varphi_a = \psi_a f$ for all $a \in A$.
A morphism from $\varphi$ to itself is called an endomorphism of $\varphi$.
Denote by $\End_{L'} \varphi$ the set of all endomorphisms of $\varphi$ over $L'$.
\end{definition}

Obviously, $\End_{L'} \varphi$ forms a subring of $L' \{ \tau \}$.
For any $a \in A$ the constant polynomial $a = a \tau^0$ is clearly an endomorphism of $\varphi$.
Therefore $\End_{L'} \varphi$ always contains $A$.
However, not all of $\varphi \in \Drin_A L$ have $A$ as their endomorphism rings.
We set $\Drin_A^0 L$ to be the set of $\varphi \in \Drin_A L$ such that $\End_{\tilde{L}} \varphi = A$.

For any non-zero prime ideal $P$ of $A$ the \textit{$P$-adic Tate module} $T_P(\varphi)$ of $\varphi$ is defined by
\[ T_P(\varphi) = \varprojlim {}_\varphi \tilde{L}[P^n]. \]
From Theorem \ref{I-tor}, if $\varphi$ has rank $r$, then $T_P(\varphi)$ is a free $A_P$-module of rank $r$, where $A_P$ is the completion of $A$ at $P$.

\section{Reduction steps}

This section is devoted to interchange the order of quantifiers to restate the main theorem in terms of Drinfeld $A$-modules defined over $L$ rather than the measure of $\Gal(L^\sep/L)^e$.

\begin{proposition}\label{rdc1}
Consider the following statements on an $A$-field $M$.
\textup{
    \begin{enumerate}[(A)]
    \item \textit{Fix a Drinfeld $A$-module $\varphi$ defined over $M$ with $\End_{\tilde{M}} \varphi = A$. Then for almost all $\sigma \in \Gal(M^\sep/M)$ there exist infinitely many non-zero prime ideals $P$ of $A$ such that $_\varphi\tilde{M}(\sigma)[P] \neq 0$.}
    \item \textit{Assume $\varphi$ is as} (A) \textit{and $e \ge 2$. Then for almost all $\sigma \in \Gal(M^\sep/M)^e$ there are only finitely many prime ideals $P$ of $A$ such that $_\varphi\tilde{M}(\sigma)[P] \neq 0$.}
    \item \textit{Assume $\varphi$ is as} (A) \textit{and fix a non-zero prime ideal $P$ of $A$. Then for almost all $\sigma \in \Gal(M^\sep/M)$ the group $_\varphi\tilde{M}(\sigma)[P^\infty]$ is finite.}
    \end{enumerate}
}
\noindent Suppose that \textup{(A)} holds for all finite extensions $M$ over an $A$-field $L$.
Then \textup{(a)} of Theorem \ref{mainthm} holds for $L$.
Similarly, \textup{(B)} implies \textup{(b)}, and \textup{(C)} implies \textup{(c)}, respectively.
\end{proposition}

\begin{proof}
Suppose that (A) holds for all finite extensions $M/L$.
Let $S$ be the set of $\sigma \in \Gal(L^\sep/L)$ which does not satisfy (a).
We need to show $\mu_L(S) = 0$.
For $\varphi \in \Drin_A^0\tilde{L}$ there is a finite extension $L'/L$ such that $\varphi$ is already defined over $L'$.
Let $L_\varphi$ be the intersection of all of such $L'$.
Then $\varphi \in \Drin_A^0L_\varphi$.
Consider the compositum of the restriction and the inclusion
\[ \rho: \Gal(L_\varphi^\sep/L_\varphi) \to \Gal(L^\sep/(L_\varphi \cap L^\sep)) \hookrightarrow \Gal(L^\sep/L). \]
Since the restriction map is isomorphic and preserves the measure, for every measurable set $T$ in $\Gal(L_\varphi^\sep/L_\varphi)$ it holds that $\mu_L(\rho(T)) = \mu_{L_\varphi}(T) / [ L_\varphi : L]_s$.
Further, let
\[ S'_\varphi = \{ \sigma' \in \Gal(L_\varphi^\sep/L_\varphi) \mid {}_\varphi\tilde{L}(\sigma')_\tor \; \text{is finite} \} \]
and $S_\varphi = \rho(S'_\varphi)$.

Take $\sigma \in S$.
There exists $\varphi \in \Drin_A^0 \tilde{L}(\sigma)$ such that $_\varphi\tilde{L}(\sigma)_\tor$ is finite.
Let $\sigma' \in \Gal(L_\varphi^\sep/L_\varphi)$ be the unique extension of $\sigma$ to $L_\varphi^\sep$.
Then $\sigma' \in S'_\varphi$, so that we have $\sigma \in S_\varphi$.
Therefore it follows that $S \subseteq \bigcup_\varphi S_\varphi$, where $\varphi$ runs over $\Drin_A^0 \tilde{L}$.
Since $\mu_{L_\varphi}(S'_\varphi) = 0$ by (A), we have $\mu_L(S_\varphi) = 0$ for each $\varphi \in \Drin_A^0 \tilde{L}$.
The set $\Drin_A^0 \tilde{L}$ is denumerable, hence $\mu_L(S) = 0$, as desired.

We can prove (B) $\Rightarrow$ (b) in the same way as above.

(C) $\Rightarrow$ (c) is also similar, but it requires more arguments.
For each $P$, let $S_P$ be the set of $\sigma \in \Gal(L^\sep/L)$ for which there is $\varphi \in \Drin_A^0 \tilde{L}(\sigma)$ with $_\varphi \tilde{L}(\sigma)[P^\infty]$ infinite.
The similar arguments to the above ones show $\mu_L(S_P) = 0$ for each $P$.
Denote by $S^{(e)}$ the set of $\sigma \in \Gal(L^\sep/L)^e$ which does not satisfy (c).
Then $S^{(e)} \subseteq \big(S^{(1)}\big)^e$, since $\tilde{L}(\sigma) \subseteq \tilde{L}(\sigma_i)$ for $1 \le i \le e$.
Taking into account $S^{(1)} = \bigcup_P S_P$, we have $\mu_L\big(S^{(1)}\big) = 0$, hence $\mu_L\big(S^{(e)}\big) = 0$.
This implies (c).
\end{proof}

In order to prove the statements in Proposition \ref{rdc1}, further reduction is convenient.
Let us introduce some notations.
Let $\varphi$ be a Drinfeld $A$-module over $L$.
Define
\[ L_I = L({}_\varphi\tilde{L}[I]), \quad G_I = \Gal(L_I/L) \]
for every non-zero ideal $I$ of $A$.
Similarly, put
\[ L_{I^\infty} = L({}_\varphi\tilde{L}[I^\infty]), \quad G_{I^\infty} = \Gal(L_{I^\infty}/L). \]

Let $G$ be an arbitrary group and $Z$ be an abelian group. Suppose that $G$ operates on $Z$.
If $G$ and $Z$ have topologies, then we also suppose that the operation of $G$ is continuous.
For given $G$, $Z$, and a positive integer $e$, denote by $S_e(G, Z)$ the set of $g = (g_1, g_2, \dots, g_e) \in G^e$ such that there is non-zero $z \in Z$ satisfying $g_i z = z$ for all $i$.
In particular, we write $S(G, Z) = S_1(G, Z)$, that is, $S(G, Z)$ is the set of $g \in G$ such that there is non-zero $z \in Z$ with $g z = z$.
It is clear that $S_e(G, Z) \subseteq S(G, Z)^e$.
We deal with the following situations:
\begin{enumerate}[(1)]
    \item $G = \Gal(L'/L)$ and $Z = {}_\varphi \tilde{L}[I]$. Here $L'$ is a Galois extension of $L$ containing $L_I$, especially $L_I$ itself or $L^\sep$. Each element of $G$ operates on $_\varphi \tilde{L}[I]$ since it is determined by polynomials with coefficients in $L$.
    \item $G = \Gal(L'/L)$ and $Z = {}_\varphi \tilde{L}[I^\infty]$, where $L'$ is a Galois extension of $L$ containing $L_{I^\infty}$, especially $L_{I^\infty}$ itself or $L^\sep$. For the same reason as above, $G$ operates on $_\varphi \tilde{L}[I^\infty]$.
    \item $G = \GL(n, R)$ and $Z = R^n$. Here $n$ is a positive integer, $R$ a commutative ring with $1$ and $\GL(n, R)$ operates on $R^n$ in the usual sense.
Notice that a matrix $g \in \GL(n, R)$ belongs in $S(\GL(n, R), R^n)$ if and only if $1$ is an eigenvalue of $g$.
\end{enumerate}

\begin{proposition}\label{rdc2} 
Consider the following statements on an $A$-field $L$ and on $\varphi \in \Drin_A^0 L$.
\textup{
    \begin{enumerate}[(A')]
    \item \textit{There exists an infinite set $\Lambda$ of non-zero prime ideals of $A$ which satisfies the following}:
        \begin{enumerate}[(\text{A}'1)]
        \item $\sum_{P \in \Lambda} \card S(G_P, {}_\varphi\tilde{L}[P]) / \card G_P = \infty.$
        \item \textit{For every distinct $P_1, P_2, \dots, P_t \in \Lambda$,}
        \[ \frac{\card S_I}{\card G_I} = \prod_{j=1}^t \frac{\card S(G_{P_j}, {}_\varphi\tilde{L}[P_j])}{\card G_{P_j}}, \]
        \textit{where $I = \prod_{j=1}^t P_j$, and}
\[ S_I = \left\{ \sigma \in G_I \;\middle|\; \sigma |_{L_{P_j}} \in S(G_{P_j}, {}_\varphi\tilde{L}[P_j]) \;\; \textit{for} \;\; 1 \le j \le t \right\}. \]
        \end{enumerate}
    \item \textit{There exists a positive constant $c$ such that}
    \[ \frac{\card S(G_P, {}_\varphi\tilde{L}[P])}{\card G_P} \le \frac{c}{N(P)} \]
    \textit{for every non-zero prime ideal $P$ of $A$. Here $N(P) = \card (A/P)$.}
    \item \textit{For every non-zero prime ideal $P$ of $A$ the set $S(G_{P^\infty}, T_P(\varphi))$ is a zero set in $G_{P^\infty}$.}
    \end{enumerate}
}
\noindent Then \textup{(A')} implies \textup{(A)} of Proposition \ref{rdc1} with $M$ replaced by $L$.
Similarly, \textup{(B')} implies \textup{(B)}, and \textup{(C')} implies \textup{(C)}, respectively.
\end{proposition}

\begin{proof}
(A') $\Rightarrow$ (A): The lifting of $S(G_P, {}_\varphi\tilde{L}[P])$ to $\Gal(L^\sep/L)$ with respect to the restriction map $\Gal(L^\sep/L) \to G_P$ coincides with $S(\Gal(L^\sep/L), {}_\varphi\tilde{L}[P])$.
Then we have
\begin{equation}
\mu\left(S(\Gal(L^\sep/L), {}_\varphi\tilde{L}[P])\right) = \frac{\card S(G_P, {}_\varphi\tilde{L}[P])}{\card G_P}
\end{equation}
for $P \in \Lambda$.
From (A'1), it follows that $\sum_{P \in \Lambda} \mu(S(\Gal(L^\sep/L), {}_\varphi\tilde{L}[P])) = \infty$.

Now take distinct $P_1, P_2, \dots, P_t \in \Lambda$.
The lifting of $S_I$ to $\Gal(L^\sep/L)$ is the intersection of $S(\Gal(L^\sep/L), {}_\varphi\tilde{L}[P_j])$, $1 \le j \le t$.
Thus we have
\begin{alignat*}{2}
\mu\left(\bigcap_{j=1}^t S(\Gal(L^\sep/L), {}_\varphi\tilde{L}[P_j])\right) &= \frac{\card S_I}{\card G_I} & \quad & \\
 &= \prod_{j=1}^t \frac{\card S(G_{P_j}, {}_\varphi\tilde{L}[P_j])}{\card G_{P_j}} & & \text{by (A'2)} \\
 &= \prod_{j=1}^t \mu\left(S(\Gal(L^\sep/L), {}_\varphi\tilde{L}[P_j])\right) & & \text{by (1).}
\end{alignat*}
This implies the $\mu$-independence of $\{ S(\Gal(L^\sep/L), {}_\varphi\tilde{L}[P]) \}_{P \in \Lambda}$.
Therefore, by Borel-Cantelli's lemma, the set
\[ S = \{ \sigma \in \Gal(L^\sep/L) \mid \sigma \in S(\Gal(L^\sep/L), {}_\varphi\tilde{L}[P]) \; \text{for infinitely many} \; P\text{'s} \} \]
has measure one in $\Gal(L^\sep/L)$.
For every $\sigma \in S$ there are infinitely many prime ideals $P$ of $A$ such that $_\varphi\tilde{L}(\sigma)[P] \neq 0$.

(B') $\Rightarrow$ (B): From (1) and (B'), we obtain
\[ \mu\left(S(\Gal(L^\sep/L), {}_\varphi\tilde{L}[P])\right) \le \frac{c}{N(P)} \]
for each non-zero prime ideal $P$ of $A$.
Therefore
\[ \mu\left(S_e(\Gal(L^\sep/L), {}_\varphi\tilde{L}[P])\right) \le \mu\left(S(\Gal(L^\sep/L), {}_\varphi\tilde{L}[P])\right)^e \le \left(\frac{c}{N(P)}\right)^e. \]
Since $e \ge 2$, the sum $\sum_P \mu(S_e(\Gal(L^\sep/L), {}_\varphi\tilde{L}[P]))$ converges.
Applying Borel-Cantelli's lemma, for almost all $\sigma \in \Gal(L^\sep/L)^e$ there are only finitely many prime ideals $P$ such that $\sigma \in S_e(\Gal(L^\sep/L), {}_\varphi\tilde{L}[P])$. 
This implies that $_\varphi\tilde{L}(\sigma)_\tor$ is finite.

(C') $\Rightarrow$ (C): It suffices to show that if $\sigma \in \Gal(L^\sep/L)$ and the group $_\varphi\tilde{L}(\sigma)[P^\infty]$ is infinite, then $\sigma|_{L_{P^\infty}} \in S(G_{P^\infty}, T_P(\varphi))$.
Since the class number $h$ of $A$ is finite, 
there is $b \in A$ such that $P^h = (b)$.
Then
\[ T_P(\varphi) = \varprojlim {}_\varphi\tilde{L}[P^n] \cong \varprojlim {}_\varphi\tilde{L}[b^m]. \]
Because each $_\varphi\tilde{L}(\sigma)[b^m]$ is finite but the union $\bigcup_{m=1}^\infty {}_\varphi\tilde{L}(\sigma)[b^m]$ is infinite, there are infinitely many $m$'s such that
\[ _\varphi\tilde{L}(\sigma)[b^m]^\ast = {}_\varphi\tilde{L}(\sigma)[b^m] \setminus {}_\varphi\tilde{L}(\sigma)[b^{m-1}] \]
is not empty.
If $x \in {}_\varphi\tilde{L}(\sigma)[b^m]^\ast$, then $\varphi_b(x) \in {}_\varphi\tilde{L}(\sigma)[b^{m-1}]^\ast$.
Namely, $\varphi_b$ maps $_\varphi\tilde{L}(\sigma)[b^m]^\ast$ into $_\varphi\tilde{L}(\sigma)[b^{m-1}]^\ast$.
In particular, $_\varphi\tilde{L}(\sigma)[b^m]^\ast$ is not empty for all $m$, hence it follows from \cite[Corollary 1.1.4]{FrJ} that $\varprojlim {}_\varphi\tilde{L}(\sigma)[b^m]^\ast$ is also not empty.
Each point of this inverse limit yields a non-zero point in $T_P(\varphi)$ which is fixed by $\sigma$.
Therefore we obtain $\sigma|_{L_{P^\infty}} \in S(G_{P^\infty}, T_P(\varphi))$, as desired.
\end{proof}

\section{Proof of the main theorem}

In this section we show (A')--(C') of Proposition 4.2 and complete the proof of the main theorem.

The proofs of Parts (A') and (B') are carried out in parallel.
By using the theorem of Pink and R\"{u}tsche about the representation on Drinfeld modules, it turns out that Parts (A') and (B') follow by estimating the density of matrices in $\GL(r, \mathbb{F}_q)$ which fix a non-zero element in $\mathbb{F}_q^r$.

\begin{lemma}
Let $r$ be a positive integer.
There exist positive constants $c_1$, $c_2$ such that
\[ c_1 q^{r^2-1} \le \card S(\GL(r, \mathbb{F}_q), \mathbb{F}_q^r) \le c_2 q^{r^2-1} \]
for all prime powers $q$.
Hence we have
\[ \frac{c'_1}{q} \le \frac{\card S(\GL(r, \mathbb{F}_q), \mathbb{F}_q^r)}{\card \GL(r, \mathbb{F}_q)} \le \frac{c'_2}{q} \]
for some constants $c'_1, c'_2 > 0$.
\end{lemma}

\begin{proof} 
The second assertion follows from the first one and the well-known fact that $\card \GL(r, \mathbb{F}_q) = (q^r -1) (q^r -q) \dotsm (q^r - q^{r-1})$. Thus it is sufficient to prove the existence of $c_1$ and $c_2$.

Set $G = \GL(r, \mathbb{F}_q)$, $V = \mathbb{F}_q^r$.
If $r=1$, then we see $S(G, V) = \{ 1 \}$, so that we take $c_1 = c_2 =1$.

Assume $r \ge 2$. Let $1 \le j \le r$.
There are $(q^r -1) (q^r - q) \dotsm (q^r - q^{j-1})$ linearly independent (ordered) $j$-tuples of $V$, and each of them generates a $j$-dimensional subspace of $V$.
Since there are $(q^j - 1) (q^j -q) \dotsm (q^j - q^{j-1})$ (ordered) bases of every $j$-dimensional subspace, the number of $j$-dimensional subspaces of $V$ is
\[ t_j = \frac{(q^r - 1) (q^r - q) \dotsm (q^r - q^{j-1})}{(q^j - 1) (q^j -q) \dotsm (q^j - q^{j-1})}. \]
List all the subspaces of dimension $j$, say $W_1^{(j)}, W_2^{(j)}, \dots, W_{t_j}^{(j)}$.
Further, let
\[ S_i^{(j)} = \left\{ g \in G \;\middle|\; g a = a \; \text{for all} \; a \in W_i^{(j)} \right\} \]
for $1 \le j \le r$, $1 \le i \le t_j$, and set
\[ B_j = \left\{ g \in G \;\middle|\; \dim_{\mathbb{F}_q} V^g = j \right\}, \quad b_j = \card B_j, \]
where $V^g = \{ a \in V \mid g a = a \}$.

For each $1 \le j \le r$, $1 \le i \le t_j$, we have
\[ \card S_i^{(j)} = (q^r -q^j) (q^r - q^{j+1}) \dotsm (q^r -q^{r-1}). \]
Now we claim
\[ \bigcup_{h=j}^r B_h = \bigcup_{i=1}^{t_j} S_i^{(j)}\]
for $1 \le j \le r$.
Indeed, both hands of the equality coincide with the set of all $g \in G$ such that $\dim_{\mathbb{F}_q} V^g \ge j$.
Thus it follows that
\begin{align*}
b_j &= \card B_j \le \rule{0em}{4ex}^\#\!\! \Biggl( \bigcup_{h=j}^r B_h \Biggr) = \rule{0em}{4ex}^\#\!\! \Biggl( \bigcup_{i=1}^{t_j} S_i^{(j)} \Biggr) \le \sum_{i=1}^{t_j} \card S_i^{(j)} \\
 &= \frac{(q^r - 1) (q^r - q) \dotsm (q^r - q^{j-1})}{(q^j - 1) (q^j -q) \dotsm (q^j - q^{j-1})} (q^r -q^j) (q^r - q^{j+1}) \dotsm (q^r -q^{r-1}) \\
 &\le \alpha_j q^{r^2-j^2}.
\end{align*}
Here $\alpha_j$ is some positive constant depending only on $r$ and $j$, not on $q$.
Since $S(G, V) = \bigcup_{h=1}^r B_h$, substituting $j=1$ shows that we can take $c_2 = \alpha_1$.

For $1 \le j \le r$ and $g \in B_j$ there are $(q^j - 1) / (q - 1)$ subspaces of dimension one in $V^g$.
Thus there are  $(q^j - 1) / (q - 1)$ $i$'s such that $g \in S_i^{(1)}$.
Note that $\{ B_j \}_{j=1}^r$ gives a partition of $S(G, V)$.
Then we find
\begin{align*}
\sum_{i=1}^{t_1} \card S_i^{(1)} &= \sum_{j=1}^r \frac{q^j - 1}{q - 1} b_j \\
 &= \sum_{j=1}^r b_j + \sum_{j=2}^r \left( \frac{q^j - 1}{q - 1} - 1 \right) b_j \\
 &= \card S(G, V) + \sum_{j=2}^r \left( \frac{q^j - 1}{q - 1} - 1 \right) b_j.
\end{align*}
Using $q \ge 2$, we see $(q^j - 1) / (q - 1) - 1 \le 2 q^{j-1}$ for $2 \le j \le r$.
Hence we estimate
\[ \sum_{j=2}^r \left( \frac{q^j - 1}{q - 1} - 1 \right) b_j \le \sum_{j=2}^r 2 q^{j-1} \cdot \alpha_j q^{r^2-j^2} \le \alpha q^{r^2-3} \]
for some constant $\alpha > 0$.
Therefore
\begin{align*}
\card S(G, V) &= \sum_{i=1}^{t_1} \card S_i^{(1)} - \sum_{j=2}^r \left( \frac{q^j - 1}{q - 1} - 1 \right) b_j \\
 &\ge \frac{(q^r - 1) (q^r - q) \dotsm (q^r - q^{r-1})}{q - 1} - \alpha q^{r^2-3},
\end{align*}
which says the existence of $c_1$.
\end{proof}

\begin{proposition}
Let $L$ be a finitely generated extension of $K$.
Then \textup{(A')} and \textup{(B')} of Proposition \ref{rdc2} hold for every $\varphi \in \Drin_A^0 L$.
\end{proposition}

\begin{proof}
Consider an embedding
\[ \rho_I: G_I = \Gal(L({}_\varphi\tilde{L}[I])/L) \hookrightarrow \GL_{A/I}({}_\varphi\tilde{L}[I]) \cong \GL(r, A/I) \]
for each non-zero ideal $I$ of $A$.
Here $r$ is the rank of $\varphi$ and the isomorphism $\GL_{A/I}({}_\varphi\tilde{L}[I]) \cong \GL(r, A/I)$ is defined under a fixed basis of $_\varphi\tilde{L}[I]$ over $A/I$.
By the theorem of Pink and R\"{u}tsche \cite[Theorem 0.1]{PR}, there exists an ideal $C$ of $A$ such that $\rho_I$ is surjective for all $I$ prime to $C$. 
Let $\Lambda$ be the set of all non-zero prime ideals prime to $C$.
Then all but finitely many prime ideals belong to $\Lambda$.
For $P \in \Lambda$, we know
\[ \frac{\card S(G_P, {}_\varphi\tilde{L}[P])}{\card G_P} = \frac{\card S(\GL(r, A/P), (A/P)^r)}{\card \GL(r, A/P)}. \]
By the above lemma, there exist $c_1$, $c_2 > 0$ such that
\[ \frac{c_1}{N(P)} \le \frac{\card S(G_P, {}_\varphi\tilde{L}[P])}{\card G_P} \le \frac{c_2}{N(P)} \]
for all $P \in \Lambda$.
This proves (A'1) and (B') of Proposition \ref{rdc2} because $\sum_P 1 / N(P)$, where $P$ ranges over all non-zero prime ideals of $A$, diverges.

To prove (A'2), it is sufficient to show that if $I$ and $J$ are ideals of $A$ such that $I$, $J$, $C$ are pairwise relatively prime, then $L_I$ and $L_J$ are linearly disjoint over $L$.
Indeed, we have
\[ L_{IJ} = L_I L_J, \quad \GL(r, A/IJ) \cong \GL(r, A/I) \times \GL(r, A/J). \]
Therefore
\begin{multline*}
[L_I L_J : L] = [L_{IJ} : L] = \card G_{IJ} = \card \GL(r, A/IJ) \\
= \card \GL(r, A/I) \card \GL(r, A/J) = \card G_I \card G_J = [L_I : L] [L_J : L].
\end{multline*}
This implies the linearly disjointness of $L_I$ and $L_J$, as desired.
\end{proof}

Now we can complete the proof of Parts (a) and (b) in the main theorem.
Let $L$ be a finitely generated extension of $K$.
Proposition \ref{rdc1} says that (A) for all finitely generated extensions $L/K$ implies (a).
From Proposition \ref{rdc2} in order to derive (A) it is sufficient to show (A') for every $\varphi \in \Drin_A^0 L$, but this is a result of the above proposition.
By the same arguments, we have (b) in Theorem \ref{mainthm}. \\

\par Next, we prove Part (c) in the main theorem.
The above discussion for Parts (a) and (b) can also apply to this part.
Therefore what we need to work out is to show (C') of Proposition \ref{rdc2} for every finitely generated extension $L/K$ and for every $\varphi \in \Drin_A^0 L$.

\begin{lemma}
Let $P$ be a non-zero prime ideal of $A$ and $r$ be a positive integer.
Then $S(\GL(r, A_P), A_P^r)$ is a zero set in $\GL(r, A_P)$.
\end{lemma}

\begin{proof} 
Put $S = S(\GL(r, A_P), A_P^r)$, $G = \GL(r, A_P)$, respectively.
$S$ consists of all matrices in $G$ having $1$ as an eigenvalue.
In particular, $S$ is closed in $G$.
For any $u \in A_P^\times$, $u S$ consists of all matrices in $G$ having $u$ as an eigenvalue.
Choose $u \in A_P^\times$ with infinite order.
We claim that if $t$ is a positive integer and $i_1, i_2, \dots, i_t$ are distinct non-negative integers, then $u^{i_1}S \cap u^{i_2}S \cap \dots \cap u^{i_t}S$ is a zero set in $G$.
The lemma is the special case where $t=1$ and $i_1=0$.

Indeed, if $t \ge r+1$ and $i_1, i_2, \dots, i_t$ are distinct non-negative integers, then $u^{i_1}S \cap u^{i_2}S \cap \dots \cap u^{i_t}S$ consists of matrices which have $u^{i_1}, u^{i_2}, \dots, u^{i_t}$ as eigenvalues.
However, every matrix in $G$ has at most $r$ eigenvalues, which implies that $u^{i_1}S \cap u^{i_2}S \cap \dots \cap u^{i_t}S$ is empty.
In particular, it is a zero set.

Now assume that the claim holds for all integers greater than $t$.
We show the claim for $t$.
Let $i_1, i_2, \dots, i_t$ be distinct non-negative integers.
Put $D = u^{i_1}S \cap u^{i_2}S \cap \dots \cap u^{i_t}S$.
Consider the translation $u^m D$ for each non-negative integer $m$.
If $m \neq m'$, then the intersection $u^m D \cap u^{m'} D$ is the intersection of at least $t+1$ sets of $u^i S$'s.
By the assumption of induction, $\mu_G(u^m D \cap u^{m'} D) = 0$.
Therefore for each positive integer $N$ we obtain
\[ \mu_G(D) = \frac{1}{N} \sum_{m=0}^{N-1} \mu_G(u^m D) = \frac{1}{N} \mu_G \left( \bigcup_{m=0}^{N-1} u^m D \right) \le \frac{1}{N}. \]
Taking $N \to \infty$, we have $\mu_G(D) = 0$.
\end{proof}

\begin{proposition}
Assume that $L$ is finitely generated over $K$. Then \textup{(C')} of Proposition \ref{rdc2} holds for every $\varphi \in \Drin_A^0 L$.
\end{proposition}

\begin{proof}
Let $\varphi \in \Drin_A^0 L$ and $P$ be a non-zero prime ideal of $A$.
Consider an embedding
\[ \rho_{P^\infty} : G_{P^\infty} = \Gal(L({}_\varphi\tilde{L}[P^\infty])/L) \hookrightarrow \GL_{A_P}(T_P(\varphi)) \cong \GL(r, A_P). \]
Here $r$ is the rank of $\varphi$ and the isomorphism $\GL_{A_P}(T_P(\varphi)) \cong \GL(r, A_P)$ is defined under a fixed basis of $T_P(\varphi)$ over $A_P$.
Denote by $G(\varphi, P)$ the image of $\rho_{P^\infty}$.
Then $\rho_{P^\infty}$ maps $S(G_{P^\infty}, T_P(\varphi))$ into $S(G(\varphi, P), A_P^r) \subseteq S(\GL(r, A_P), A_P^r)$.
The result of Pink \cite[Theorem 0.1]{Pi} asserts that $G(\varphi, P)$ has finite index in $\GL(r, A_P)$.
Hence it suffices to show that the measure of $S(\GL(r, A_P), A_P^r)$ is zero in $\GL(r, A_P)$, but this is already proved in the above lemma.
\end{proof}

\ \\
\small{
Department of Mathematics, Tokyo Institute of Technology

2-12-1, O-okayama, Meguro-ku, Tokyo 152-8551, Japan

\textit{E-mail address} : \texttt{asayama.t.aa@m.titech.ac.jp}
}

\begin{thebibliography}{9999}
\bibitem[FyJ]{FyJ}
    G. Frey and M. Jarden,
    \textit{Approximation theory and the rank of abelian varieties over large algebraic fields},
    Proc.\ London Math.\ Soc.\ \textbf{28} (1974), 112--128.
\bibitem[FrJ]{FrJ}
    M. D. Fried and M. Jarden,
    \textit{Field Arithmetic}, third edition,
    Springer-Verlag, Berlin, 2008.
\bibitem[GJ1]{GJ1}
    W.-D.\ Geyer and M. Jarden,
    \textit{Torsion points of elliptic curves over large algebraic extensions of finitely generated fields},
    Israel J. Math.\ \textbf{31} (1978), 257--297.
\bibitem[GJ2]{GJ2}
    W.-D.\ Geyer and M. Jarden,
    \textit{Torsion of abelian varieties over large algebraic fields},
    Finite Fields Appl.\ \textbf{11} (2005), 123--150.
\bibitem[JJ1]{JJ1}
    M. Jacobson and M. Jarden,
    \textit{On torsion of abelian varieties over large algebraic extensions of finitely generated fields},
    Mathematika \textbf{31} (1984), 110--116.
\bibitem[JJ2]{JJ2}
    M. Jacobson and M. Jarden,
    \textit{On torsion of abelian varieties over large algebraic extensions of finitely generated fields}: \textit{erratum},
    Mathematika \textbf{32} (1985), 316.
\bibitem[JJ3]{JJ3}
    M. Jacobson and M. Jarden,
    \textit{Finiteness theorems for torsion of abelian varieties over large algebraic fields},
    Acta Arith.\ \textbf{98} (2001), 15--31.
\bibitem[JP]{JP}
    M. Jarden and S. Petersen,
    \textit{Torsion of abelian varieties over large algebraic extensions of} $\mathbb{Q}$,
    Nagoya Math.\ J. \textbf{234} (2019), 46--86.
\bibitem[Pi]{Pi}
    R. Pink,
    \textit{The Mumford-Tate conjecture for Drinfeld-modules},
    Publ.\ Res.\ Inst.\ Math.\ Sci.\ \textbf{33} (1997), 393--425.
\bibitem[PR]{PR}
    R. Pink and E. R\"{u}tsche,
    \textit{Adelic openness for Drinfeld modules in generic characteristic},
    J. Number Theory \textbf{129} (2009), 882--907.
\bibitem[Po]{Po}
    B. Poonen,
    \textit{Local height functions and the Mordell-Weil theorem for Drinfeld modules},
    Compositio Math.\ \textbf{97} (1995), 349--368.
\bibitem[Ro]{Ro}
    M. Rosen,
    \textit{Number Theory in Function Fields},
    Springer-Verlag, New York, 2002.
\bibitem[Wa]{Wa}
    J. T.-Y.\ Wang,
    \textit{The Mordell-Weil theorems for Drinfeld modules over finitely generated function fields},
    Manuscripta Math.\ \textbf{106} (2001), 305--314.
\bibitem[Zy]{Zy}
    D. Zywina,
    \textit{Abelian varieties over large algebraic fields with infinite torsion},
    Israel J. Math.\ \textbf{211} (2016), 493--508.
\end{thebibliography}
\end{document}